\documentclass[11pt,oneside,a4paper]{amsart}

 \newtheorem{thm}{Theorem}[section]
 \newtheorem{cor}[thm]{Corollary}
 \newtheorem{lem}[thm]{Lemma}
 \newtheorem{prop}[thm]{Proposition}
 \theoremstyle{definition}
 \newtheorem{defn}[thm]{Definition}
 \theoremstyle{remark}
 \newtheorem{rem}[thm]{Remark}
 
 \numberwithin{equation}{section}

\usepackage{bbm,times,epsfig}
\usepackage{amsfonts,amscd,amsmath,amssymb,amsthm}

\renewcommand{\phi}{\varphi}

\DeclareSymbolFont{SY}{U}{psy}{m}{n}
\DeclareMathSymbol{\emptyset}{\mathord}{SY}{'306}

\DeclareMathOperator{\Ran}{Ran}

\DeclareMathOperator{\ran}{ran}
\DeclareMathOperator{\ind}{ind}
\DeclareMathOperator{\dimk}{dim\,ker}

\DeclareMathOperator{\Tr}{Tr}

\newcommand{\Jcup}{J_\cup}

\DeclareMathSymbol{\newtimes}{\mathbin}{SY}{'264}

\newcommand{\barr}[1]{ \overline {#1} }
\newcommand{\clos}[1]{ {#1}^\text{\,cl}}

\newcommand{\HH}{\wt \cH} %  !!!!!
\newcommand{\wt}{\widetilde}

\renewcommand{\phi}{\varphi}

\newcommand{\R}{\mathbb{R}}

\newcommand{\C}{\mathbb{C}}
\newcommand{\Z}{\mathbb{Z}}
\newcommand{\N}{\mathbb{N}}

\newcommand{\cA}{{\mathcal A}}
\newcommand{\cB}{{\mathcal B}}

\newcommand{\cF}{{\mathcal F}}
\newcommand{\cH}{{\mathcal H}}

\newcommand{\cK}{{\mathcal K}}

\newcommand{\cL}{{\mathcal L}}

\newcommand{\cO}{{\mathcal O}}
\newcommand{\cP}{{\mathcal P}}

\newcommand{\cS}{{\mathcal S}}
\newcommand{\cT}{{\mathcal T}}

\newcommand{\eps}{\varepsilon} %  !!!!!
\newcommand{\la}{\lambda}
\newcommand{\Om}{\Omega}

\newcommand\beqn{\begin{equation}}
\newcommand\neqn{\end{equation}}

\newcommand\cPfin{{\mathcal P}_\text{fin}}

\newcommand{\1}{\mathbbm 1}

\DeclareMathAlphabet{\Ma}{U}{msa}{m}{n}
\DeclareMathAlphabet{\Mb}{U}{msb}{m}{n}

\DeclareMathAlphabet{\Meuf}{U}{euf}{m}{n}
\DeclareSymbolFont{ASMa}{U}{msa}{m}{n}
\DeclareSymbolFont{ASMb}{U}{msb}{m}{n}
\DeclareMathSymbol{\hrist}{\mathord}{ASMa}{"16}
\DeclareMathSymbol{\varkappa}{\mathalpha}{ASMb}{"7B}
\DeclareMathSymbol{\CrPr}{\mathord}{ASMb}{"6F}

\def\got#1{\Meuf{#1}}
\def\ot #1.{{\got{#1}}}

\DeclareSymbolFont{SY}{U}{psy}{m}{n}
\DeclareMathSymbol{\emptyset}{\mathord}{SY}{'306}

\begin{document}

%-------------------------------------------------------------------------
% editorial commands: to be inserted by the editorial office
%
%\firstpage{1} \volume{228} \Copyrightyear{2004} \DOI{003-0001}
%\seriesextraline{This is the Concrete Title of this Book\br H.E. R and S.T.C. W, Eds.}
%\firstpage{1}

\title[F\o lner sequences]
 {F\o lner sequences in operator theory \\ 
   and operator algebras}

%----------Author 1

\author{Pere Ara}
\address{Department of Mathematics,
  Universitat Aut\`onoma de Barcelona, 08193 Bellaterra (Bar\-ce\-lona), Spain}
\email{para@mat.uab.cat}

\author{Fernando Lled\'o}
\address{Department of Mathematics, University Carlos~III Madrid,
  Avda.~de la Universidad~30, E-28911 Legan\'es (Madrid), Spain
  and Instituto de Ciencias Matem\'{a}ticas (CSIC - UAM - UC3M - UCM)}
\email{flledo@math.uc3m.es}

\author{Dmitry V. Yakubovich}
\address{Departamento de Matem\'{a}ticas,
Universidad Autonoma de Madrid, Cantoblanco 28049 (Ma\-drid) Spain
\enspace and
\phantom{r} Instituto de Ciencias
Matem\'{a}ticas (CSIC - UAM - UC3M - UCM)}
\email{dmitry.yakubovich@uam.es}

\thanks{
The first-named author was partially supported by
the project MTM2011-28992-C02-01
of the Spanish Ministry of Science and Innovation (MICINN), also  
supported by FEDER, and by the Comissionat per Universitats i
Recerca de la Generalitat de Catalunya. The second-named
and third-named authors were partly supported by the MICINN
and FEDER projects MTM2012-36372-C03-01 and, respectively,  
MTM2008-06621-C02-01, MTM2011-28149-C02-1.
The last two authors also acknowledge the support by the ICMAT  
Severo Ochoa project SEV-2011-0087.
}


\date{\today}

\begin{abstract}
The present article is a review of recent developments concerning the notion of
F\o lner sequences both in operator theory and operator algebras. We also give a new
direct proof that any essentially normal operator has an increasing F\o lner
sequence $\{P_n\}$ of non-zero finite rank projections that strongly converges to $\1$.
The proof is based on Brown-Douglas-Fillmore theory.
We use F\o lner sequences to analyze the class of finite operators introduced by Williams in 1970.
In the second part of this article we examine
a procedure of approximating any amenable trace on a unital and separable C*-algebra by tracial states 
$\mathrm{Tr}(\cdot P_n)/\mathrm{Tr}(P_n)$
corresponding to a F\o lner sequence and apply this method to improve spectral approximation results
due to Arveson and B\'edos. The article concludes with the analysis of C*-algebras admitting a non-degenerate
representation which has a F\o lner sequence or, equivalently, an
amenable trace. We give an abstract characterization of these algebras in terms of unital
completely positive maps and define F\o lner
C*-algebras as those unital separable C*-al\-ge\-bras that satisfy these
equivalent conditions. This is analogous to Voiculescu's abstract
characterization of quasidiagonal C*-algebras.
\end{abstract}

%%% ----------------------------------------------------------------------
\maketitle
%%% ----------------------------------------------------------------------
%\tableofcontents
%%%%%%%%%%%%%%%%%%%%%%%%%%%%%%%%%%%%%%%%%%%%%%%%%%%%%%%%%%%%%%%%%%%%%%%%%%%%%%%%%%%%%%%%%
\section{Introduction}
%%%%%%%%%%%%%%%%%%%%%%%%%%%%%%%%%%%%%%%%%%%%%%%%%%%%%%%%%%%%%%%%%%%%%%%%%%%%%%%%%%%%%%%%%

In their beginnings the single operator theory and the theory of operator algebras
were a common subject and shared many techniques. As an example recall that
von Neumann algebras were introduced (as {\em rings of operators})
in 1929 by von Neumann in his second paper on spectral theory \cite{vNeumann29}.
In recent times, however, each
of these theories has developed own elaborated techniques which in many cases
remain unknown to experts of the other area. Nevertheless single operator theory
and the theory of operator algebras have also had fruitful and important interactions ever since.
Brown, Douglas and Fillmore's theory was motivated by the classification of essentially normal
operators and ended with the introduction of the $Ext$-group as a fundamental invariant for operator
algebras. Finally, Voiculescu's work on quasidiagonality also shows the importance of the dialog
between these communities (cf.~\cite{bDavidson96,Voiculescu93,voicu91,Voiculescu76}).

In more recent times operator al\-ge\-bra techniques, in particular exact C*-al\-ge\-bras,
have also been used to solve Herrero's approximation problem
for quasidiagonal operators (cf.~\cite{Brown01}). Moreover, operator algebras have shown to
be a useful tool in order to address problems in spectral approximation:
given a sequence of linear operators
$\{T_n\}_{n\in\N}$ in a complex separable Hilbert space $\cH$ that approximates
an operator $T$ in a suitable sense, a natural question is how do the spectral
characteristics of $T$ (the spectrum, spectral measures, numerical ranges, pseudospectra etc.)
relate with those of $T_n$ as $n$ grows.
(Excellent books that include a large number of examples and references
are, e.g., \cite{bChatelin83,bAhues01}. See also
\cite{Boettcher00,bHagen01} for the application of C*-algebra techniques in
numerical analysis.)
Arveson's seminal series of articles
\cite{Arveson93,Arveson94,ArvesonIn94} on this topic were directly inspired by
Szeg\"o's classical approximation theorem for Toeplitz operators.
Among other interesting results, Arveson
gave conditions that guarantee that the essential spectrum of a
large class of selfadjoint
operators $T$ may be recovered from the sequence of eigenvalues
of certain finite dimensional compressions $T_n$. These results were then
refined by B\'edos who systematically applied the concept of F\o lner sequence
of non-zero finite rank projections
to spectral approximation problems (see \cite{Bedos94,Bedos95,Bedos97}
as well as \cite{bHagen01,pLledo11}; for a precise definition of F\o lner sequence
and additional results we refer to Section~\ref{sec:basic}).
It is stated in Section~7.2 of \cite{bHagen01} that SeLegue also considered
Szeg\"o-type theorems for Toeplitz operators in the context of C*-algebras.
Hansen extends some of the mentioned results
to the case of unbounded operators (cf.~\cite[\S~7]{Hansen08}; see also
\cite{Hansen11} for recent developments in the non-selfadjoint case).
Brown shows in \cite{Brown06b} that abstract
results in C*-algebra theory can be applied to compute spectra of important
operators in mathematical physics like
almost Mathieu operators or periodic magnetic Schr\"odinger operators on graphs.

In the last two decades, the relation between spectral approximation problems and F\o lner sequences
for non-selfadjoint and non-normal operators has been also
explored, see for instance \cite{Widom94,Tilli99,Boettcher05,Roch07}.

The aim of this article is to present in a single publication recent operator theoretic
and operators algebraic results that involve the notion of F\o lner sequences for
operator. F\o lner sequences were introduced
in the context of operator algebras by Connes in Section~V of his seminal
paper \cite{Connes76} (see also \cite[Section~2]{ConnesIn76}).
This notion is an algebraic analogue of F\o lner's characterization
of amenable discrete groups and was used by Connes as an essential tool in the
classification of injective type~II$_1$ factors.
Part of the material of this paper is taken from \cite{AraLledo12, pLledoYak12}.
There is though a new and complete proof that any essentially normal operator
has a proper F\o lner sequence (cf.~Subsection~\ref{subsec:ess-normal}) which,
in our opinion, is interesting in its own right.
The proof is based on the absorbing property for direct sums
stated in Proposition~\ref{localization} and pure operator theoretic arguments including
Brown-Douglas-Fillmore theory.

In this article we will present
(with the exception of Subsection~\ref{subsec:ess-normal})
only short proofs that improve the comprehension of the statement or that contain useful
techniques. For more difficult and elaborate arguments we will refer to the original publications.
Section~\ref{sec:operators} is completed with the analysis of the relations between the class
of finite operators (introduced by Williams in \cite{Williams70})
and the notion of F\o lner sequence. It shown that F\o lner sequences for
operators provide a very useful and natural tool to analyze this class
of operators. In the last section we will study the role of F\o lner sequences in operator
algebras. First we review the relation between F\o lner sequences for a unital and separable
C*-algebra $\cA$ and amenable traces. In particular, we present
an approximation procedure for amenable traces in terms of F\o lner sequences of projections
\cite[Theorem 6.1]{Ozawa04} (see also \cite[Theorem 6.2.7]{bBrown08}). We apply this method
in Theorem~\ref{teo:apply} to extend a spectral approximation result for scalar spectral measures
in the spirit of Arveson and B\'edos. In Subsection~\ref{sbsc:Foe-alg} we give finally
an abstract characterization of unital separable
C*-algebras $\cA$ admitting a non-degenerate representation
$\pi$ on a Hilbert space such that there is a F\o lner sequence for
$\pi(\cA)$ or, equivalently, such that $\pi(\cA)$ has an amenable
trace (see Theorem \ref{thm:charFoe}). We conclude with a brief
discussion of C*-algebras that can also be related to a given
F\o lner sequence and that appear naturally in the context of spectral approximation
problems. In the last section we summarize some of the main relations and differences in the 
analysis of F\o lner sequences for single operators and for abstract C*-algebras.
\smallskip

\textbf{Notation:}
We will denote by $\cL(\cH)$ the C*-algebra of bounded and linear
operators on the complex separable Hilbert space $\cH$,
and by $\cK(\cH)$ the ideal of compact operators on $\cH$.
Next, $\cPfin(\cH)$ is the set of all non-zero
finite rank orthogonal projections on $\cH$ and  $[A,B]:=AB-BA$
stands for the commutator of two operators $A,B\in\cL(\cH)$. We denote by
$\mathrm{Tr}(\cdot)$ the standard trace on $\cL(\cH)$ and
by $\mathrm{tr}(\cdot)$ the unique tracial state on a matrix algebra
$M_{n}(\C)$, $n\in\N$.

%%%%%%%%%%%%%%%%%%%%%%%%%%%%%%%%%%%%%%%%%%%%%%%%%%%%%%%%%%%%%%%%%%%%%%%%%%%%%%%%%%%%%%%%%
\section{Basic properties of F\o lner sequences for operators}\label{sec:basic}
%%%%%%%%%%%%%%%%%%%%%%%%%%%%%%%%%%%%%%%%%%%%%%%%%%%%%%%%%%%%%%%%%%%%%%%%%%%%%%%%%%%%%%%%%

The notion of F\o lner sequences for operators has its origins in
group theory. Recall that a discrete countable group $\Gamma$ is
said to be \textit{amenable} if it has an invariant mean, i.e., there is a positive
linear functional $\psi$ on the von Neumann algebra\footnote{We
identify here each $f\in \ell^\infty(\Gamma)$ with the multiplication
operator with $f$ on the Hilbert space $\ell_2(\Gamma)$.}
$\ell^\infty(\Gamma)$ with norm one such
that
\[
  \psi(\gamma f)=\psi(f)\;,\quad \gamma\in\Gamma\;,\quad f\in \ell^\infty(\Gamma)\;,
\]
where $(\gamma f)(\gamma_0):=f(\gamma^{-1}\gamma_0)$.
Abelian groups, finite groups and their extensions are amenable.
A F{\o}lner sequence for the group $\Gamma$ is a sequence of non-empty finite
subsets $\Gamma_i\subset\Gamma$ that satisfy
\begin{equation}\label{eq:foelner-g}
\lim_{i}
\frac{|(\gamma \Gamma_i)\triangle
\Gamma_i|}{|\Gamma_i|} =0\qquad\text{for all}\quad \gamma\in\Gamma \,,
\end{equation}
where $\triangle$ denotes the symmetric difference and $|X|$ is the
cardinality of a set $X$. Then, $\Gamma$ has a F{\o}lner
sequence if and only if $\Gamma$ is amenable (see, e.g., Chapter~4 in
\cite{bPaterson88}).
An analysis of different properties of approximability of a group by finite groups and
their relation to amenability has been undertaken in the review \cite{Vershik82}.

The counterpart of the preceding definition in the context of operators
is given next. First we need to recall that if $T\in\cL(\cH)$, then
$\|T\|_p$, $p=1,2,\dots$, is its norm in the Schatten-von Neumann
class.

\begin{defn}\label{def:Foelner}
Let $\mathcal{T}\subset\cL(\cH)$ be a set of operators. A sequence of non-zero
finite rank orthogonal projections $\{P_n\}_{n\in \N}\subset\cPfin(\cH)$ on $\cH$
is called a {\em F{\o}lner sequence for $\mathcal{T}$} if
\begin{equation}\label{eq:F1}
\lim_{n\to\infty} \frac{\|T P_n-P_n T\|_2}{\|P_n\|_2} = 0\;\;,\quad T\in\cT\; .
\end{equation}
If the F\o lner sequence $\{P_n\}_{n\in \N}$ satisfies, in addition, that it is increasing
and converges strongly to $\1$, then we say it is a {\em proper F\o lner sequence for $\cT$}.
\end{defn}

The existence of a F\o lner sequence has already important structural consequences, see for instance
Proposition \ref{pro:hypertrace} and Corollary \ref{cor:fin=Fol_seq} below.
Notice, however, that proper F\o lner sequences are important in the context of spectral approximation
in the spirit of works \cite{Widom94, Tilli99, Roch07} and others.

In the preceding definition we have not specified any structure on the set of operators $\cT$.
Typically, $\cT$ will be a single operator or a concrete C*-algebra, realized in a Hilbert space
$\cH$. The next result collects some immediate consequences of the definition of a
F\o lner sequence for operators. Part (ii) is shown in Lemma~1 of \cite{Bedos97}
(see also \cite[Proposition~2.1]{pLledoYak12}).

\begin{prop}\label{pro-1}
Let $\cT\subset\cL(\cH)$ be a set of operators and $\{P_n\}_{n\in \N}\subset\cPfin(\cH)$
a sequence of non-zero finite rank orthogonal projections. Then we have
\begin{itemize}
 \item[(i)] $\{P_n\}_{n\in \N}$ is a F\o lner sequence for $\cT$ if and only if it is a F\o lner sequence for
$C^*\big(\,\cT\,,\,\,\1 \big)$, where $C^*(\cdot)$ is the C$^*$-algebra generated by its argument. Moreover,
$\{P_n\}_{n\in \N}$ is a proper F\o lner sequence for $\cT$ if and only if it is a proper F\o lner sequence for
\[
  C^*\big(\,\cT\,,\, \cK(\cH)\,,\,\1 \big)\;.
\]
\item[(ii)]
$\{P_n\}_{n\in \N}$ is a F{\o}lner sequence for $\cT$ if and only if
the following condition holds:
\begin{equation}\label{F1}
\lim_{n\to\infty} \frac{\|T P_n-P_n T\|_1}{\|P_n\|_1} = 0\; ,\quad T\in\cT\;.
\end{equation}
If $\cT$ is a self-adjoint set (i.e., $\cT^*=\cT$), then
$\{P_n\}_{n\in \N}$ is a F{\o}lner sequence for $\cT$ if and only if
for all $T\in\cT$:
\begin{equation}\label{F2}
\lim_{n\to\infty} \frac{\|(I-P_n)T P_n\|_p}{\|P_n\|_p} = 0,\qquad
p\in\{1,2\}\;.
\end{equation}
\item[(iii)]
Let $T\in\cL(\cH)$ and
$\{Q_n\}_n\subset\cPfin(\cH)$ be such that the sequence $\{\dim Q_n\}$ is unbounded and
\[
\lim_{n\to\infty} \frac{\|T Q_n-Q_nT\|_2}{\|Q_n\|_2} = 0\;.
\]
Then there exists a proper F\o lner sequence for $T$.
\end{itemize}
\end{prop}

Proper F\o lner sequences for operators have the following characteristic
absorbing property for direct sums:

\begin{prop}\label{localization}
Let $\cH$ and $\cH'$ be separable Hilbert spaces with $\dim\cH=\infty$.
If $T$ has a proper F\o lner sequence, then $T\oplus X\in\cL(\cH\oplus\cH')$
has a proper F\o lner sequence for {\em any} $X\in\cL(\cH')$.
\end{prop}
\begin{proof}
Let $\{P_n\}_{n\in \N}$ be a  proper F\o lner sequence for $T$ and since
the sequence of projections is increasing we may assume
that $\dim P_n\cH\geq n^2\;$. Let $\{e_i\mid i\in\N\}$ be an orthonormal
basis of $\cH'$ and denote by $Q_n$ the orthogonal projection onto
span$\{e_1,\dots,e_n\}\subset\cH'$.
Then the following calculation shows that $\{P_n\oplus Q_n\}_n$ is a proper F\o lner sequence
for $T\oplus X$, $X\in\cL(\cH')$:
\begin{eqnarray*}
\frac{ \left\| \big[ T\oplus X,P_n\oplus Q_n\big]\right\|_2^2}{\|P_n\oplus Q_n\|^2_2}
  &=& \frac{\left\|[T,P_n]\right\|^2_2+\left\|[X,Q_n]\right\|_2^2}{\|P_n\|^2_2+n}\\
  &\leq &  \frac{\left\|[T,P_n]\right\|^2_2}{\|P_n\|^2_2}
          +\frac{4\left\|Q_n\right\|^2_2\;\left\|X\right\|^2}{n^2+n} \\
  &=& \frac{\left\|[T,P_n]\right\|^2_2}{\|P_n\|^2_2}
          +4\left\|X\right\| \frac{n}{n^2+n} \quad\mathop{\longrightarrow}\limits_{n\to\infty}\quad 0\,.
          \qquad  \qedhere
\end{eqnarray*}
\end{proof}

Next, we mention some first operator algebraic consequences related to the existence of
F\o lner sequences. For this we need to recall the following notion:
\begin{defn}
A state $\tau$ on the unital C*-algebra $\cA\subset\cL(\cH)$
(i.e., a positive and normalized linear functional on $\cA$) is called an
\textit{amenable trace}
if there exists a state $\psi$ on $\cL(\cH)$ such that
$\psi\upharpoonright\cA=\tau$ and
\begin{equation*}
\psi(X A) = \psi(A X)\;,\quad X\in \cL(\cH)\;,\;A\in\cA\,.
\end{equation*}
The state $\psi$ is also referred in the literature as a
\textit{hypertrace for} $\cA$.
\end{defn}
Note that an amenable trace is really a trace on $\cA$ (i.e.,
$\tau(AB)=\tau(BA)$, $A,B\in\cA$). We also refer to \cite{Brown06a,Ozawa04} for a
thorough description of the relations of amenable traces and F\o lner
sequences to other important areas like, e.g., Connes' embedding problem.
Hypertraces are the algebraic analogue of the invariant mean
on groups mentioned at the beginning of this section.
Later we will need the following standard result.
(See \cite{Connes76,ConnesIn76} for
the original statement and more results in the context of operator
algebras; see also \cite{Bedos95,AraLledo12} for additional results in the
context of C*-algebras related to the existence of a hypertrace.)
\begin{prop}\label{pro:hypertrace}
  Let $\cA\subset\cL(\cH)$ be a separable unital C*-algebra. Then
 $\cA$ has a F\o lner sequence if and only if $\cA$ has an amenable trace.
\end{prop}

Finally, we also mention the following useful results in the context of single operator theory.
We need to introduce first the following definition.
\begin{defn}
We say that $T\in \cL(\cH)$ is
{\em finite block reducible} if $T$ has a non-trivial finite-dimensional reducing
subspace, i.e., there is an orthogonal decomposition $\cH=\cH_0\oplus \cH_1$
which reduces $T$ and where $\cH_0$ is finite dimensional and non-zero.
\end{defn}

The following two propositions are technical and we refer to Section~3 in \cite{pLledoYak12}
for a complete proof.

\begin{prop}\label{prop:decomp}
Let $T= T_0\oplus \wt T$ on $\cH=\cH_0\oplus\HH$, where $\dim \cH_0 < \infty$.
Then, $T$ has a proper F\o lner sequence if and only if $\wt T$ has a proper F\o lner sequence.
\end{prop}

Note that in the reverse implication of Proposition~\ref{pro:hypertrace} the sequence of projections
does not have to be a proper F\o lner sequence in the sense of Definition~\ref{def:Foelner}.
In fact, one can easily construct the following counterexample: consider
a finite block reducible operator $T= T_0\oplus T_1$
on the Hilbert space $\cH=\cH_0\oplus\cH_1$, with $1\le \dim \cH_0<\infty$ and
where $T_1$ has no F\o lner sequence (examples of these type of operators
will be given in Section~\ref{sbsc:strg-NF}). Then, one can show
that $C^*(T\,,\,\1)$ has a
hypertrace (see Williams theorem in Subsection~\ref{subsec:Williams}) and
by Proposition~\ref{pro:hypertrace} it has
a F\o lner sequence also. The obvious choice of F\o lner sequence is the
constant sequence $P_n=\1_{\cH_0}\oplus 0$, $n\in\N$, which trivially
satisfies (\ref{eq:F1}) for $T$. But $T$ cannot have a proper F\o lner sequence,
because $T_1$ has no F\o lner sequence by Proposition~\ref{prop:decomp}.

The following proposition clarifies the relation between F\o lner sequences and proper F\o lner sequences
in the context of operator theory. In a sense the difference between F\o lner sequence and proper F\o lner
sequence can only appear if the operator is finite block reducible.

\begin{prop}\label{prop:splitting}
Let $T\in\cL(\cH)$ and suppose that $T P-P T\not=0$
for all $ P\in \cPfin(\cH)$. If there is a
F\o lner sequence of projections $\{P_n\}_n\subset \cPfin(\cH)$ of a constant rank,
then $T$ has a proper F\o lner sequence.
\end{prop}

%%%%%%%%%%%%%%%%%%%%%%%%%%%%%%%%%%%%%%%%%%%%%%%%%%%%%%%%%%%%%%%%%%%%%%%%%%%%%%%%%%%%%%%%%
\section{F\o lner sequences in operator theory}\label{sec:operators}
%%%%%%%%%%%%%%%%%%%%%%%%%%%%%%%%%%%%%%%%%%%%%%%%%%%%%%%%%%%%%%%%%%%%%%%%%%%%%%%%%%%%%%%%%

Using a classical result by Berg that states that any normal
operator can be expressed as a sum of a diagonal operator and a
compact operator (cf.~\cite[Section~II.4]{bDavidson96}) it is
immediate that any normal operator has a proper F\o lner sequence.
In the next subsection we will address the question of existence of
proper F\o lner sequences for an important class of non-normal
operators. We will also explore the structure of operators that have
no proper F\o lner sequence. Finally we will show the strong link
between the class of finte operators (in the sense of Williams
\cite{Williams70}) and the notion of F\o lner sequence.

%%%%%%%%%%%%%%%%%%%%%%%%%%%%%%%%%%%%%%%%%%%%%%%%%%%%%%%%%%%%%%%%%%%%%%%%%%%%%%%%%%%%%%%%%
\subsection{Essentially normal operators}\label{subsec:ess-normal}
%%%%%%%%%%%%%%%%%%%%%%%%%%%%%%%%%%%%%%%%%%%%%%%%%%%%%%%%%%%%%%%%%%%%%%%%%%%%%%%%%%%%%%%%%

In this subsection we give a proof of the fact that any essentially
normal operator has a proper F\o lner sequence. The proof below is
from an earlier version of \cite{pLledoYak12}. In this
reference we present a stronger statement, namely that any
essentially hyponormal operator has a proper F\o lner sequence by using
different techniques (see Theorem~5.1 in
\cite{pLledoYak12})\footnote{An operator $T\in \cL(\cH)$ is
called \textit{hyponormal} if its self-commutator $[T^*,T]$ is
nonnegative. $T$ is called \textit{essentially hyponormal} if the
image in the Calkin algebra $\cL(\cH)/\cK(\cH)$ of $[T^*,T]$ is a
nonnegative element. Any essentially normal operator is essentially
hyponormal (see, e.g., \cite[Chapter~4]{bConway91} or the review
\cite{Volberg90} for additional results).}.

Nevertheless in our opinion the present direct proof is interesting in itself and the reasoning
is completely different from that in \cite{pLledoYak12}.
The proof below is based on the absorbing property for direct sums
given in Proposition~\ref{localization} and pure operator theoretic arguments including
Brown-Douglas-Fillmore theory.

We begin showing that the unilateral shift $S$ has a canonical proper F\o lner sequence.
In fact, define $S$ on $\cH:=\ell^2(\N_0)$ by $Se_i:=e_{i+1}$, where
$\{e_i\mid i=0,1,2,\dots\}$ is the canonical basis of $\cH$
and consider for any $n$ the orthogonal projection $P_n$ onto span$\{e_i\mid i=0,1,2,\dots, n\}$.
Then
\[
 \big\|[P_n,S]\big\|_2^2=\sum_{i=1}^\infty \Big\|[P_n,S]e_i\Big\|^2=\|e_{n+1}\|^2=1
\]
and
\[
\frac{\big\|[P_n,S]\big\|_2}{\| P_n\|_2}=\frac{1}{\sqrt{n+1}}\;\mathop{\longrightarrow}\limits_{n\to\infty} \;0\;.
\]

Next we recall some definitions and facts concerning essentially normal operators.
Details and additional references can be found, e.g., in \cite{BrownIn73,Fillmore72}; see also
\cite{Davidson-ess}  for an excellent brief up-to day account of essential normality and the
Brown--Douglas--Fillmore theory.
An operator $T\in\cL(\cH)$ is called {\em essentially normal}
if its self-commutator is a compact operator, i.e.,~if $[T,T^*]\in\cK(\cH)$. If $\rho$
is the quotient map from $\cL(\cH)$ onto the Calkin algebra $\cL(\cH)/\cK(\cH)$, then $T$ is
essentially normal if and only if $\rho(T)$ is normal in the Calkin algebra.
The unilateral shift $S$ mentioned above is a standard example of
an essentially normal operator,
since its self-commutator is a rank~$1$ projection.
We recall that an operator $F\in\cL(\cH)$ is
called \textit{Fredholm} if its range $\ran F$ is
closed and both $\ker F$ and $(\ran F)^\perp$ are finite dimensional.
The index of a Fredholm operator $F$ is defined as
\[
 \ind(F)=\dimk F-\dim (\ran F)^\perp\;.
\]
\textit{The essential spectrum} of an operator $T$ is
\[
\sigma_{\mathrm{ess}}(T):= \{\lambda\in\C \mid T-\lambda\1\;\;\mathrm{is~not~Fredholm} \}\;.
\]
If $F$ is Fredholm and $K$ is compact, then $F+K$ is Fredholm and
$\ind(F+K)=\ind(F)$. Finally, $F\in\cL(\cH)$ is Fredholm if and only if
$\rho(F)$ is invertible in the Calkin algebra. Therefore, the essential
spectrum of any $T\in\cL(\cH)$ coincides with the spectrum of $\rho(T)$.
We refer to Section~I.8 in \cite{Shubin01} for an accessible exposition of Fredholm operators.

We will need later the following standard facts:
\begin{prop}\label{lem:invert}
Let $\{ T_n\}_{n\in \N}$ be a sequence of bounded operators in
$\cL(\cH_n)$.
 \begin{itemize}
  \item[(i)] Assume $\sup_{n}\left\{\|T_n\|\right\}<\infty$ and define the bounded operator
   $\widehat{T}=\oplus_n T_n$ on $\oplus_n \cH_n$. Then, $\widehat{T}$ is invertible if and only if each $T_n$ is invertible and
\[
\sup_{n}\left\{\|T_n^{-1}\|\right\}<\infty\;.
\]
  \item[(ii)] $\sigma_{\mathrm{ess}}(T_1\oplus T_2)=\sigma_{\mathrm{ess}}(T_1)\cup\sigma_{\mathrm{ess}}(T_2)$.
  \item[(iii)] If $T_1, T_2$ are Fredholm operators, then
               $\mathrm{ind}(T_1\oplus T_2)=\mathrm{ind}(T_1)+\mathrm{ind}(T_2)$.
 \end{itemize}
\end{prop}

The proof of the main result of this subsection
is based on the existence of operators having specific spectral
properties. In what follows, for a
subset $\Om$ of the complex plane, we denote by $\clos \Om$ its closure and put
$$
\barr \Om=\{\bar z \mid  z\in \Om\}\;.
$$
We will use
the space $R^2(\Omega)$,
defined as the closure in $L^2(\Omega)$ (with the Lebesgue measure) of the set of rational functions
with poles off $\clos\Om$ (see, e.g., Chapter~1 in \cite{bConway91}).

We need to recall here also some other standard notions in operator theory.
An operator $T\in\cL(\cH)$ is called {\em finitely multicyclic}
if there are finitely many vectors $g_1,\dots,g_m\in\cH$ such that the span of the set
\[
\{ u(T)g_i\mid 1\leq i\leq m\;,\;\;   u\;\;\mathrm{rational~function~with~poles~off~}\sigma(T)\}
\]
is dense in $\cH$. The vectors $g_1,\dots,g_m$ are called a cyclic set of vectors. If $T$ is finitely
multicyclic and $m$ is the smallest number of cyclic vectors, then $T$ is called $m$-multicyclic.

For the reader's convenience, we recall the following classical result due to
Berger and Shaw and which we will use several times. For details we refer to the original
article \cite{Berger-Shaw-73} or to Section~IV.2 in \cite{bConway91}.

\begin{thm}[Berger-Shaw]
Suppose $T$ is an $m$-multicyclic hyponormal operator. Then its self-commutator
$[T^*, T]$ is of trace class and the canonical trace satisfies
\begin{equation*}
 \Tr \big([T^*, T]\big)
        \leq \frac{m}{\pi}\,\mathrm{area}(\sigma(T))\; ,
\end{equation*}
where $\sigma(T)$ denotes the spectrum of $T$.
\end{thm}

\begin{lem}\label{lem:construction}
Let $\Omega$ be an open, bounded and connected subset of $\C$. Then, the multiplication operator on
$R^2(\Omega)$ given by
\[
 \left(M_\Omega f\right)(z):=z\,f(z)\;,\quad f\in R^2(\Omega)\;,
\]
satisfies the following properties:
\begin{itemize}
  \item[(i)] $\sigma\left(M_\Omega\right) = \clos\Om$ and $\|M_\Omega\|=\max_{z\in\clos\Om}\;\{|z|\}$.
  \item[(ii)] $\sigma_{\mathrm{ess}}\left(M_\Omega\right) \subset \partial\Omega$ and
\[
 \mathrm{ind}\,(M_\Omega-\lambda\1) =
 \begin{cases}
  0\, , \quad  &\lambda\not\in\clos\Om    \\
  -1\, , \quad &\lambda\in\Omega\;.
 \end{cases}
\]
 \item[(iii)] $\|(M_\Omega-\lambda)^{-1}\| = \left(\mathrm{dist}\left(\lambda,\clos\Om \right)\right)^{-1}
              \;,\;\lambda\not\in\clos\Om\;.$
 \item[(iv)] $M_\Omega$ is a hyponormal operator.\footnote{In fact a stronger property holds: $M_\Omega$ is a subnormal
 operator (i.e., the restriction of a normal operator to an invariant subspace); see \cite{bConway91}.}
 \item[(v)] The self-commutator $[M_\Omega^*,M_\Omega]$ is a trace-class operator and
\begin{equation}\label{eq:Berger-Shaw}
 \Tr \big([M_\Omega^*,M_\Omega]\big)
        \leq \frac{1}{\pi}\,\mathrm{area}(\Omega)\; .
\end{equation}
\end{itemize}
\end{lem}
\begin{proof}
It is a standard fact that $R^2(\Om)$ consists of analytic functions on $\Omega$
and that for any $\la\in\Om$, the evaluation functional
$f\mapsto f(\la)$ is bounded on $R^2(\Om)$ (see, e.g., Section~II.7 in \cite{bConway91}).
Parts (i) and (iii) follow from standard properties of the multiplication operator. (Note that for
$\lambda\not\in\clos\Om$ the function $(z-\lambda)^{-1}$ is bounded and analytic in $\Omega$.)
To prove (ii), it suffices to observe that $\ker (M_\Omega-\lambda)=\{0\}$ for any $\la\in \C$,
that $\Ran (M_\Omega-\lambda)=R^2(\Om)$ for $\la\notin \clos\Om$ and that
\[
\Ran (M_\Omega-\lambda)=\big\{f\in R^2(\Om)\mid  f(\la)=0\big\} \quad \text{for  $\la\in\Om$}\,.
\]
This gives the formula for the index stated above.

To prove (iv) note that $M_\Omega^*f=Q_R (\overline{z}\,f)$, $f\in R^2(\Omega)$, where
$Q_R$ denotes the orthogonal projection from $L^2(\Omega)$ onto $R^2(\Omega)$. Therefore
$\|M_\Omega^* f\|\leq \|M_\Omega f\|$, $f\in R^2(\Omega)$, which implies that
$M_\Omega$ is a hyponormal. Finally,
by the definition of $R^2(\Omega)$, the constant function $1$ is cyclic for $M_\Omega$, so that
$M_\Omega$ is $1$-multicyclic. Hence we can
apply Berger-Shaw Theorem to conclude
that $[M_\Omega^*,M_\Omega]$ is a trace-class operator
and that the inequality stated above holds.
\end{proof}

\begin{defn}\label{def:sp-picture}
Let $T,R\in\cL(\cH)$ be essentially normal operators. We say that $T$ and $R$
have the same spectral picture the following two conditions hold:
\begin{itemize}
 \item[(i)] $\sigma_{\mathrm{ess}}(T)=\sigma_{\mathrm{ess}}(R)=:X$
 \item[(ii)] $\mathrm{ind}(T-\lambda\1)=\mathrm{ind}(R-\lambda\1)\;$, $\lambda\not\in X$.
\end{itemize}
\end{defn}

\begin{thm}\label{main1}
Any essentially normal operator $T\in\cL(\cH)$ has a proper F\o lner sequence.
\end{thm}
\begin{proof}
(i) The first step of the proof uses the following
classical result of the Brown-Douglas-Fillmore
theory (see \cite[Section~V]{Brown73} or \cite[Theorem~11.1]{BrownIn73}).
Let $T,R\in\cL(\cH)$ be essentially normal, then we have that $T=U(R+K)U^*$
for some compact operator $K$ and some unitary $U$ if and only if operators $T$ and $R$
have the same spectral picture (cf.~Definition~\ref{def:sp-picture}).
Therefore to prove that $T$ has a proper F\o lner sequence it will be
enough to construct an essentially normal operator
$R$ with a proper F\o lner sequence and having the same spectral picture as $T$. Indeed, if
$\{P_n\}_n$ is a proper F\o lner sequence for $R$, then by Proposition~\ref{pro-1}~(i)
it is also a proper F\o lner sequence for $R+K$ for any compact operator $K$
and therefore $\widehat{P}_n:=UP_nU^*$ is a proper F\o lner sequence for
$T=U(R+K)U^*$.

(ii) Given the essentially normal operator $T$,
the construction of an essentially normal operator
$R$ with the same spectral picture as $T$ and having a proper F\o lner sequence goes as follows.
The set $X:=\sigma_{\mathrm{ess}}(T)$ is a closed
and bounded subset of $\C$, so that we consider its
decomposition
\[
 \C\setminus X :=\mathop{\cup}\limits_{j\in J} \Omega_j
\]
into open, connected and disjoint sets; here $J\subset\N$ is a set of indices. The index function
$\cup_j \Omega_j \ni\lambda\mapsto\mathrm{ind}\,(T-\lambda\1)$ is continuous and therefore
constant on each connected component $\Omega_j$.

We denote for $\lambda\in\Omega_j$
the index by $n_j:=\mathrm{ind}\,(T-\lambda\1)\in\Z$, and put
\[
J_- := \{j\in J\mid n_j < 0\}, \enspace
J_+ := \{j\in J\mid n_j>0\}, \enspace
\Jcup = J_- \cup J_+.
\]
These sets of indices may be  finite or infinite.

To construct $R$, first take any normal operator $N$ on an infinite dimensional
Hilbert space $\cK$ such that $\sigma_{\mathrm{ess}}(N)=X$.
(A concrete example can be constructed as follows: put $\cH=\ell^2(\N)$ and let
$\{d_n\}_{n\in\N}$ be a dense sequence of points in $X$. Any isolated point in
$X$ is repeated infinitely many times. Then the diagonal operator
$N:=\mathrm{diag}\,(\{d_n\}_n)$ is normal and $\sigma_{\mathrm{ess}}(N)=X$.)
Since $N$ is normal we have $\mathrm{ind}\,(N-\lambda\1)=0$, $\lambda\not\in X$.

Second, for any bounded $\Omega_j$, $j\in J_+$ (i.e., $n_j > 0$)
we consider the operator $M_j:=M_{\Omega_j}$ on $R^2(\Omega_j)$
as in Lemma~\ref{lem:construction} and that satisfies the properties (i)-(iv).
If $j\in J_-$, then we put $M_j:=M_{\barr \Om_j}$ on $R^2(\barr \Om_j)$.
Define the Hilbert spaces $\cK_j:=\oplus^{n_j} R^2(\Omega_j)$ for $n_j>0$ and
$\cK_j:=\oplus^{|n_j|} R^2(\barr \Om_j)$ for $n_j<0$.
Next, we construct on $\cK_j$ the operator
\[
S_j:=
 \begin{cases}
 \mathop{\oplus}\limits_1^{|n_j|} M_j^*   \;,\quad & \mathrm{if}\quad n_j < 0\;, \\
\mathop{\oplus}\limits_1^{\vphantom{l} n_j} M_j   \;,\quad      & \mathrm{if}\quad n_j > 0 \;.
 \end{cases}
\]
From Proposition~\ref{lem:invert}~(iii) and
Lemma~\ref{lem:construction}~(ii) we have
$\mathrm{ind}\,(S_j-\lambda\1)=n_j$ for any $\lambda\in\Omega_j$.
Then we consider the operator
\[
 \widehat{S}:=\Big(\mathop{\oplus}\limits_{j\in \Jcup } S_j\Big)
 \quad\mathrm{on}\quad
 \widehat{\cK}:=\mathop{\oplus}\limits_{j\in \Jcup} \cK_j
\]
and, finally, we put
\[
 R:=N\oplus \widehat{S}\in\cL(\cK\oplus \widehat{\cK})\;.
\]

(iii) The last part of the proof consists in showing
that $R$ satisfies all the required properties.

Since $N$ is normal it has a proper F\o lner sequence.
By the absorbing property of proper F\o lner sequences for
direct sums stated in Proposition~\ref{localization}
we conclude that $R$ has a proper F\o lner sequence too.

Next we show that $R$ has the same spectral picture as the given operator $T$.
For this purpose we prove first
that $\sigma_{\mathrm{ess}}(\widehat{S})\subset X$ and that for $\la $ in $\Om_j$,
$\mathrm{ind}\,(\widehat{S}-\lambda\1)=n_j$.
Assume that $\lambda\notin X$. Then
$\lambda\in \Omega_k$ for some index $k\in J$. If $k \notin \Jcup$,
put $d:=\inf_{j\in \Jcup}\left\{\mathrm{dist}\left(\lambda, \clos \Om_j\right)\right\}$.
In this case $d>0$. From Lemma~\ref{lem:construction}~(iii) we obtain
\[
\left\|(S_j-\lambda\1)^{-1}  \right\|= \frac{1}{\mathrm{dist}\big(\lambda,\clos\Om_j\big)}
                                     \leq \frac{1}{d}\, ,  \qquad j\in \Jcup\, .
\]
We conclude that the operator $\widehat{S}-\lambda\1$ is invertible
(recall Proposition~\ref{lem:invert}~(i)), hence it is Fredholm of index $0$ and
$\lambda\not\in\sigma_{\mathrm{ess}}(\widehat{S})$.

Now consider the case when $\lambda\in\Omega_k$, where
$k \in \Jcup$. Then we may consider the decomposition
\[
 \widehat{S}-\lambda\1=(S_k-\lambda\1)\oplus\big(\mathop{\oplus}\limits_{j\ne k} S_j-\lambda\1\big)\;.
\]
The same argument as before shows that $\mathop{\oplus}\limits_{j\ne k} (S_j-\lambda\1)$ is invertible, hence
Fredholm of index $0$. By construction of $S_k$ (see Lemma~\ref{lem:construction}~(ii)) and
by Proposition~\ref{lem:invert}~(iii) we conclude that $\lambda\not\in\sigma_{\mathrm{ess}}(\widehat{S})$ and that
$\mathrm{ind}\,(\widehat{S}-\lambda\1)=n_k$, for any $\lambda\in\Omega_k$.
Therefore we have that $\sigma_{\mathrm{ess}}(\widehat{S})\subset X$.

From the properties of the normal operator $N$ constructed in step
(ii), we have $\sigma_{\mathrm{ess}}(N)=X$. Using now
Proposition~\ref{lem:invert}~(ii) we conclude that
\[
\sigma_{\mathrm{ess}}(R) = \sigma_{\mathrm{ess}}(N) \cup \sigma_{\mathrm{ess}}(\widehat{S})
                          =X\;.
\]
Moreover, we have for any $\lambda\in\Omega_j$
\[
\mathrm{ind}\,(R-\lambda\1)=0+n_j= \mathrm{ind}\,(T-\lambda\1)\;,
\]
and we have shown that $T$ and $R$ have the same spectral picture.

Finally, we still have to show that $R$ is essentially normal, i.e.,
that the self-commutator of $R$ is compact. For this note that
\[
 [R^*,R]=0 \oplus [\widehat{S}^*,\widehat{S}]\;.
\]
We need to consider two cases: if the index set $J_\cup$ is finite, then
by Lemma~\ref{lem:construction}~(v) the operator $\widehat{\cS}$ is
trace class, hence $R$ is essentially normal.
Note that $\mathop{\cup}\limits_{j\in \Jcup} \Omega_j$
is bounded.
Therefore, if the set $\Jcup$ has infinite cardinality, then
we have, in addition,
\begin{equation}\label{area}
\lim_{\Jcup\ni j\to\infty}\mathrm{area}(\Omega_j) = 0 \;.
\end{equation}
Consider the partial direct sum
$
 \widehat{S}_N:= \mathop{\oplus}\limits_{j\in \Jcup, \, j\le N} S_j
$.
Applying again Lemma~\ref{lem:construction}~(v) we get
\begin{eqnarray*}
\left\| [\widehat{S}^*,\widehat{S}]- [\widehat{S}_N^*,\widehat{S}_N]  \right\|
    &=& \Big\|
        \mathop{\oplus}\limits_{j\in \Jcup , \, j> N}
         [S_j^*,S_j] \Big\|
         \; = \;
        \sup_{j\in \Jcup , \, j> N}
        \left\| [M_j^*,M_j] \right\| \\
   &\leq &  \frac{1}{\pi}\sup_{j\in \Jcup ,  \, j> N}
          \mathrm{area}(\Omega_j)\to 0 \qquad \text{as} \enspace N\to \infty
\end{eqnarray*}
(see Eq.~\eqref{area}).
Since $[\widehat{S}_N^*,\widehat{S}_N]$ is a trace-class operator, it follows that
the self-commutator $[R^*,R]$ can be approximated in norm by
trace-class operators, hence it is compact and we conclude that $R$ is essentially normal.
\end{proof}

\begin{cor}\label{hyponormal}
If $T\in\cL(\cH)$ is an $m$-multicyclic hyponormal operator, then $T$ has a proper F\o lner sequence.
\end{cor}
\begin{proof}
By Berger-Shaw Theorem
it follows that the self-commutator $[T^*,T]$ is trace-class and, therefore,
$T$ is essentially normal and the assertion follows from Theorem~\ref{main1}.
\end{proof}

We conclude this subsection mentioning that any quasinormal operator
(i.e., any operator $Q$ that commutes with $Q^*Q$) has a proper
F\o lner sequence. Recall also that an operator $T$ on $\cH$ is called subnormal
if there is a normal operator $N$ acting on a Hilbert space
$\widetilde{\cH}$ containing $\cH$ such that $\cH$ is invariant for $N$ and
$T$ is the restriction of $N$ to $\cH$. It can also be shown that any subnormal
operator has a proper F\o lner sequence.
See \cite{pLledoYak12} for details and also
Chapter~II in \cite{bConway91} for the relations between these classes of
operators.

%%%%%%%%%%%%%%%%%%%%%%%%%%%%%%%%%%%%%%%%%%%%%%%%%%%%%%%%%%%%%%%%%%%%%%%%%%%%%%%%%%%%%%%%%
\subsection{Finite operators}\label{subsec:Williams}
%%%%%%%%%%%%%%%%%%%%%%%%%%%%%%%%%%%%%%%%%%%%%%%%%%%%%%%%%%%%%%%%%%%%%%%%%%%%%%%%%%%%%%%%%

In this subsection we study the class of finite operators introduced by Williams in
\cite{Williams70} and their relation to proper F\o lner sequences.
(See also \cite{HerreroIn94}.)

We begin recalling the main definition and known results.

\begin{defn}
$T\in\cL(\cH)$ is called a \textit{finite operator} if
\[
 0\in\Big(W\left([T,X]\right)\Big)^{\mathrm{cl}}\quad\mathrm{for~all}\quad X\in\cL(\cH)\;,
\]
where $W(T)$ denotes the numerical range of the operator $T$, i.e.,
\[
W(T)=\{\langle Tx,x\rangle\mid x\in\cH\;\;,\;\;\|x\|=1 \}\;,
\]
and where the $(\cdot)^{\mathrm{cl}}$ means the closure of the corresponding subset in $\C$.
\end{defn}

We collect in the following theorem some standard results due to Williams about the class of finite operators
(cf.~\cite{Williams70}).

\begin{thm}[Williams]
An operator $T\in\cL(\cH)$ is finite if and only if
$C^*(T,\1)$ has an amenable trace. The class of finite operators
is closed in the operator norm and contains all finite block reducible
operators.
\end{thm}

It follows that the norm closure of the set of all
finite block reducible operators is contained in the class of finite operators.
Combining Williams' Theorem with Proposition~\ref{pro:hypertrace}, we get the following fact.

\begin{cor}
\label{cor:fin=Fol_seq}
For any operator $T\in \cL(H)$, the following properties are equivalent:
\begin{itemize}
\item[(i)] $T$ is finite;

\item[(ii)] $T$ has a F\o lner sequence;

\item[(iii)] $C^*(T, \1)$ has an amenable trace.
\end{itemize}
\end{cor}

The next result shows the strong link between finite operators and proper F\o lner sequences.
We include the proof, because it is short and illustrative (cf.~\cite[Theorem~4.1]{pLledoYak12}).

\begin{thm}\label{th:williams-foe}
Let $T\in\cL(\cH)$. Then, $T$ is a finite operator if and only if $T$ is finite block
reducible or $T$ has a proper F\o lner sequence.
\end{thm}
\begin{proof}
(i) If $T$ is finite block reducible, then $T$ is a finite operator (cf.~\cite{Williams70}).
Moreover, if $T$ has a proper F\o lner sequence, then the C*-algebra $C^*(T,\1)$ has the same proper F\o lner sequence
and, by Proposition~\ref{pro:hypertrace}, it also has an amenable trace. Then,
by Williams' theorem (see also Theorem~4 in \cite{Williams70}) we conclude that $T$
is finite.

(ii) To prove the other implication, assume $T$ is a finite
operator. We consider several cases. If there exists a (non-zero)
$P\in\cPfin(\cH)$ such that $[T,P]=0$, then $T$ is finite block
reducible. Consider next the situation where $[T,P]\not=0$ for all
$P\in\cPfin(\cH)$. Since $T$ is finite we can use Williams' Theorem
to conclude that $C^*(T,\1)$ has an amenable trace. Applying
Proposition~\ref{pro:hypertrace} (see also Theorem~1.1 in
\cite{Bedos95}) we conclude that there exists a F\o lner sequence of
non-zero finite rank projections $\{P_n\}_n$, i.e., we have
\[
 \lim_{n\to\infty} \frac{\|[T,P_n]\|_2}{\|P_n\|_2} = 0\;.
\]
(Note that $P_n$ is not necessarily a proper F\o lner sequence in the sense of
Definition~\ref{def:Foelner}.) Two cases may appear:
if $\mathrm{dim}\,P_n\cH\leq m$ for some $m\in\N$, then choose a subsequence with constant rank
and by Proposition~\ref{prop:splitting} we conclude that $T$ has a proper F\o lner sequence.
If the dimensions of $P_n\cH$ are not bounded, then from Proposition~\ref{pro-1}~(iii) we also
have that $T$ has a proper F\o lner sequence.
\end{proof}

%%%%%%%%%%%%%%%%%%%%%%%%%%%%%%%%%%%%%%%%%%%%%%%%%%%%%%%%%%%%%%%%%%%%%%%%%%%%%%%%%%%%%%%%%
\subsection{Strongly non-F\o lner operators}\label{sbsc:strg-NF}
%%%%%%%%%%%%%%%%%%%%%%%%%%%%%%%%%%%%%%%%%%%%%%%%%%%%%%%%%%%%%%%%%%%%%%%%%%%%%%%%%%%%%%%%%

In the present subsection we study the operators with no F\o lner sequence. For this we introduce the
following notion of operator that is far from having a non-trivial finite dimensional reducing
subspace.

\begin{defn}\label{def:strongNfol}
Let $\cH$ be an infinite dimensional Hilbert space
and $T$ an operator on $\cH$. We will say that $T$  is {\em strongly non-F\o lner}
if there exists an $\eps >0$ such that all projections $P\in\cPfin(\cH)$ satisfy
\[
\frac{\|T P-P T\|_2}{\|P\|_2}\ge \eps\;.
\]
\end{defn}

The following result shows the structure of operators with no proper F\o lner sequence.
Its proof is long and technical and we refer to Section~3 in \cite{pLledoYak12}
for details.

\begin{thm}
\label{thm:str-foln}
Let $T\in\cL(\cH)$ with $\dim \cH=\infty$. Then $T$
has no proper F\o lner sequence if and only if $T$ has an orthogonal sum
representation $T=T_0\oplus \wt T$ on $\cH=\cH_0\oplus\HH$, where $\dim \cH_0 < \infty$
and $\wt T$ is strongly non-F\o lner.
\end{thm}

Next we mention some concrete examples of strongly non-F\o lner operators.
We will use the amenable trace that appears in
Proposition~\ref{pro:hypertrace} as an obstruction.
Recall the definition of the Cuntz algebra $\cO_n$ (cf.~\cite{Cuntz77,bDavidson96}):
it is the universal C*-algebra generated by $n\geq 2$
non-unitary isometries $S_1,\ldots,S_n$ with the property that their final
projections add up to the identity, i.e.,
\begin{equation}\label{RangeProj1}
 \sum_{k=1}^{n} S_k S_k^*=\1\,.
\end{equation}
This condition implies in particular that the range projections
are pairwise orthogonal, i.e.,
\begin{equation}
\label{Cuntz}
 S_l^* S_k=\delta_{lk}\1\,.
\end{equation}
It is easy to realize the Cuntz algebra on the complex Hilbert space $\ell_2$
of square summable sequences.

\begin{prop}\label{no-Foelner}
The Cuntz algebra $\cO_n$, $n\ge 2$, is singly generated and its generator
is strongly non-F\o lner.
\end{prop}
\begin{proof}
By Corollary~4 (or Theorem~9) in \cite{Olsen76} any Cuntz algebra
$\cO_n$, $n\geq 2$, has a single generator $C_n$, i.e.,
$\cO_n=C^*(C_n)$. We assert that $C_n$ is strongly non-F\o lner.
Indeed, assume that, to the contrary, it is not; then
by Corollary~\ref{cor:yes-no} (ii), $C_n$ is finite. By
Corollary~\ref{cor:fin=Fol_seq},
it would follow that $\cO_n=C^*(C_n)$ has an amenable trace $\tau$.
But this gives a contradiction since applying $\tau$ to the equations
\eqref{RangeProj1} and \eqref{Cuntz}
we obtain $n=1$.
\end{proof}

Other examples of a strongly non-F\o lner operators can be obtained from
the proof of Theorem~5 in \cite{Halmos54}. It is also worth mentioning that Corollary~4 in
\cite{Bunce76} gives an example of a strongly non-F\o lner operator generating a type~$II_1$ factor.

Theorem~\ref{th:williams-foe} allows to divide the class of bounded linear operators into the
following mutually disjoint subclasses summarized in the following table:

\begin{table}[h]
\centering % used for centering table
  \begin{tabular}{| c || c | c | }
    \hline
    \phantom{phantom}   &  Operators with a proper   &    Operators with no proper \\[0.1ex]
    \phantom{phantom}   &  F\o lner sequence         &    F\o lner sequence          \\
                           \hline\hline
Finite block reducible     &             $\mathcal{W}_{0+}$       &                  $\mathcal{W}_{0-}$    \\ % [2mm]
\hline
Non finite block reducible &            $\mathcal{W}_{1+}$        &          $\mathcal{S}$     %%   $\mathcal{W}_{1-}$
\\ %[2mm]
    \hline
  \end{tabular}
\vskip.3cm
%Table 1
%\vskip.3cm
% \label{tab:1} % is used to refer this table in the text
 \caption{} % title of Table
\end{table}

Finally, we conclude this analysis with the following immediate consequences:

\begin{cor}\label{cor:yes-no}
Let $T\in\cL(\cH)$. Then
\begin{itemize}
 \item[(i)] $T$ is a finite operator if and only if $T$ is in one of the following mutually
disjoint classes: $\mathcal{W}_{0+}$, $\mathcal{W}_{0-}$, $\mathcal{W}_{1+}$.%% (see Table~1 in the Introduction).
 \item[(ii)] $T$ is not a finite operator (i.e., it is of class $\mathcal{S}$) if and only if $T$ is strongly non-F\o lner.
 \item[(iii)] The class of strongly non-F\o lner operators is open and dense in $\cL(\cH)$.
\end{itemize}
\end{cor}
\begin{proof}
The characterization of finite operators and its complement stated in (i) and (ii)
follows from Theorem~\ref{th:williams-foe}
and Williams' theorem.
To prove part (iii) we use that the class of finite
operators is closed and nowhere dense (cf.~\cite{Herrero89}).
Therefore the set of strongly non-F\o lner operators is an
open and dense subset of $\cL(\cH)$.
\end{proof}

As a summary let us mention that proper F\o lner sequences for operators provide a useful
and natural tool to analyze the class of finite operators. To illustrate this with an
example note that the preceding corollary already implies that the class of finite operators is
closed in $\cL(H)$.

%%%%%%%%%%%%%%%%%%%%%%%%%%%%%%%%%%%%%%%%%%%%%%%%%%%%%%%%%%%%%%%%%%%%%%%%%%%%%%%%%%%%%%%%%
\section{F\o lner sequences in operator algebras}\label{sec:algebras}
%%%%%%%%%%%%%%%%%%%%%%%%%%%%%%%%%%%%%%%%%%%%%%%%%%%%%%%%%%%%%%%%%%%%%%%%%%%%%%%%%%%%%%%%%

We start the analysis of F\o lner sequences in the context of operator algebras stating
some approximation results for amenable traces. We will apply them to spectral approximation
problems of scalar spectral measures. In the final part of this section we will give
an abstract characterization in terms of unital completely positive maps of C*-algebras
admitting a faithful essential representation which has a F\o lner sequence or, equivalently, an
amenable trace.

%%%%%%%%%%%%%%%%%%%%%%%%%%%%%%%%%%%%%%%%%%%%%%%%%%%%%%%%%%%%%%%%%%%%%%%%%%%%%%%%%%%%%%%%%
\subsection{Approximations of amenable traces}\label{subsec:approx}
%%%%%%%%%%%%%%%%%%%%%%%%%%%%%%%%%%%%%%%%%%%%%%%%%%%%%%%%%%%%%%%%%%%%%%%%%%%%%%%%%%%%%%%%%

Part~(i) of the following result is a standard weak*-compactness argument. Part~(ii)
is known to experts (see, e.g., Exercise~6.2.6 in \cite{bBrown08}) or
\cite{AraLledo12} for a complete proof).

\begin{prop}\label{exercise}
Let $\cA\subset\cL(\cH)$ be a unital separable C*-algebra.
\begin{itemize}
\item[(i)] If $\cA$ has a F\o lner sequence
$\{P_n\}_n$, then $\cA$ has an amenable trace.

\item[(ii)]
Assume that $\cA\cap\cK(\cH)=\{0\}$, and let $\tau$ be an amenable
trace on $\cA$. Then $\cA$ has a F\o lner sequence $\{P_n\}_n$
satisfying
\begin{equation}\label{eq:conv-trace}
\tau(A)=\lim_{n\to\infty}\frac{\mathrm{Tr}(AP_n)}{\mathrm{Tr}(P_n)}\;,\quad A\in\cA \;,
\end{equation}
where $\mathrm{Tr}$ denotes the canonical trace on $\cL(\cH)$.
\end{itemize}
\end{prop}

We will now present an application of
Proposition~\ref{exercise}~(ii) to obtain an approximation
result for scalar spectral measures. For this we need
to recall from \cite{Bedos97} the definition of Szeg\"o
pairs for a concrete C*-algebra $\cA\subset\cL(\cH)$. This notion
incorporates the good spectral approximation behavior of scalar
spectral measures of selfadjoint elements in $\cA$ and is motivated
by Szeg\"o's classical approximation results mentioned in the introduction.

Let $\cA$ be a unital C*-algebra acting on $\cH$ and let $\tau$ be
a tracial state on $\cA$. For any self\-adjoint element $T\in\cA$
we denote by $\mu_T$ the spectral measure associated with the trace
$\tau$ of $\cA$. Consider a sequence $\{P_n\}_n$ of non-zero finite
rank projections on $\cH$ and write the corresponding (selfadjoint)
compressions as $T_n:=P_n T P_n$. Denote by $\mu_T^n$ the
probability measure on $\R$ supported on the spectrum of $T_n$,
i.e., for any $T=T^*\in\cA$ we have
\[
 \mu_T^n(\Delta):=\frac{N_T^n(\Delta)}{\|P_n\|_1}\;,\quad \Delta\subset\R\quad \mathrm{Borel}\;,
\]
where $N_T^n(\Delta)$ is the number of eigenvalues of $T_n$
(multiplicities counted) contained in $\Delta$. We say that $\left(
\{P_n\}_n \,,\,\tau\right)$ is a {\it Szeg\"o pair} for $\cA$ if
$\mu_T^n\to \mu_T$ weakly for all selfadjoint elements $T\in\cA$,
i.e.,
\[
 \lim_{n\to\infty}
  \frac{1}{d_n}
  \Big( f(\lambda_{1,n})+\dots+f(\lambda_{d_n,n})\Big) =\int f(\lambda) \, d\mu_T(\lambda)
  \;,\quad f\in C_0(\R)  \;,
\]
where $d_n=\|P_n\|_1$ is the dimension of the $P_n\cH$ and
$\{\lambda_{1,n},\dots,\lambda_{d_n,n}\}$ are the eigenvalues
(repeated according to multiplicity) of $T_n$.

By \cite[Theorem 6~(i), (ii)]{Bedos97}, if $\left( \{P_n\}_n
\,,\tau\right)$ is a Szeg\"o pair for $\cA$, then $\{P_n\}_n$
must be a F\o lner sequence for $\cA$, $\tau$ must be an amenable
trace, and equation (\ref{eq:conv-trace}) must hold for every $A\in
\cA$. Proposition~\ref{exercise}~(ii) allows one to complete {\it any}
amenable trace $\tau$ on $\cA$ with a F\o lner sequence so that the
pair $\left( \{P_n\}_n \,,\,\tau\right)$ is a {\it Szeg\"o pair}
for $\cA$, as follows. The proof of the following result requires
the construction of an increasing sequence of operators that approximates
simultaneously the corresponding commutator and the amenable trace.
We refer to Theorem~3.2 in \cite{AraLledo12} for details

\begin{thm}\label{teo:apply}
Let $\cA$ be a unital, separable C*-algebra acting on a separable Hilbert
space $\cH$, and assume that $\cA \cap \cK (\cH)= \{ 0 \}$.  If
$\tau$ is an amenable trace on $\cA$, then there exists a proper
F\o lner sequence $\{P_n\}_n$ such that $\left( \{P_n\}_n
\,,\,\tau\right)$ is a Szeg\"o pair for $\cA$.
\end{thm}

\begin{rem}
We conclude this subsection recalling that an
important step in the proof of the Arveson-B\'edos spectral approximation
results mentioned in the introduction is the compatibility between the choice of
the F\o lner sequence in the Hilbert space and the amenable trace.
In fact, if a unital and separable concrete
C*-algebra $\cA\subset\cL(\cH)$ has an amenable trace $\tau$
and $\{P_n\}_n$ is a F\o lner sequence of non-zero
finite rank projections for $\cA$ it is needed that
the projections approximate the amenable trace in the following
natural sense
\begin{equation}\label{eq:approx-trace}
 \tau (A)=\lim_{n\to\infty} \frac{{\mathrm{Tr}}(AP_n)}{\mathrm{Tr}(P_n)}\;,\quad A\in \cA\;.
\end{equation}

Now given $\cA\subset\cL(\cH)$ with an amenable trace $\tau$
it is possible to construct a F\o lner sequence in different ways.
As observed by B\'edos in \cite{Bedos95} one way to obtain a
F\o lner sequence $\{P_n\}$ for $\cA\subset\cL (\cH )$ is
essentially contained in \cite{Connes76,ConnesIn76}.
In these articles Connes adapts the group theoretic methods by Day and
Namioka to the context of operators. Using this technique one loses
track of the initial amenable trace $\tau$, in the sense that the sequence
$\{ P_n\}$ does not necessarily satisfy (\ref{eq:approx-trace}). To avoid
this problem one may assume in addition that $\cA$ has a unique tracial
state. This is sufficient to guarantee a good spectral approximation
behavior of relevant examples like almost Mathieu
operators, which are contained in the irrational rotation algebra
(cf.~\cite{bBoca01}).

In contrast with the previous method, the
construction of a F\o lner sequence given in
\cite[Theorem 6.1]{Ozawa04} (see also \cite[Theorem 6.2.7]{bBrown08})
allows one to approximate the original trace as in
Eq.~(\ref{eq:approx-trace}). In the precedent theorem
it was crucial to use this method
to prove a spectral approximation result in the spirit
of Arveson and B\'edos, but removing the hypothesis of a unique trace
(compare Theorem~\ref{teo:apply} with
\cite[Theorem~1.3]{Bedos95} or \cite[Theorem~6~(iii)]{Bedos97}
and the formulation in p.~354 of
\cite{Arveson94}).
\end{rem}

%%%%%%%%%%%%%%%%%%%%%%%%%%%%%%%%%%%%%%%%%%%%%%%%%%%%%%%%%%%%%%%%%%%%%%%%%%%%%%%%%%%%%%%%%
\subsection{F\o lner C*-algebras}\label{sbsc:Foe-alg}
%%%%%%%%%%%%%%%%%%%%%%%%%%%%%%%%%%%%%%%%%%%%%%%%%%%%%%%%%%%%%%%%%%%%%%%%%%%%%%%%%%%%%%%%

The existence of a F\o lner sequence for a set of operators $\mathcal{T}$ is a
weaker notion than quasidiagonality. Recall that a set
of operators $\mathcal{T}\subset\mathcal{L}(\mathcal{H})$ is said to
be quasidiagonal if there exists an increasing sequence of finite-rank
projections $\{P_n\}_{n\in \N}$ converging strongly to $\1$
and such that
\begin{equation}\label{QD}
\lim_{n}\|T P_n-P_n T\|=0\;,\quad T\in\mathcal{T}\;.
 \end{equation}
(See, e.g., \cite{Halmos70,Voiculescu93} or Chapter~16 in \cite{bBrown08}.)
The existence of proper F\o lner sequences can be understood as a
quasidiagonality condition, but relative to the growth of the
dimension of the underlying spaces. It can be easily shown that if
$\{P_n\}_n$ quasidiagonalizes a family of operators $\cT$, then this
sequence of non-zero finite rank orthogonal projections is also a
proper F\o lner sequence for $\cT$.
The unilateral shift is a basic example that shows the difference between
the notions of proper F\o lner sequences and quasidiagonality.
It is a well-known fact that the unilateral shift $S$
is not a quasidiagonal operator. (This was shown by Halmos in \cite{Halmos68}; in
fact, in this reference it is shown that $S$ is not even quasitriangular.)
In the setting of abstract C*-algebras it can also be shown that a C*-algebra
containing a non-unitary isometry is not quasidiagonal
(see, e.g., \cite{BrownIn04,bBrown08}).

In \cite{voicu91}, Voiculescu
characterized abstractly quasidiagonality for unital separable
C*-algebras in terms of unital completely positive
(u.c.p.) maps\footnote{Recall that in this context a linear map
$\varphi\colon\cA\to\cB$ between unital C*-algebras $\cA$, $\cB$
is called {\em unital completely positive (u.c.p.)}, if $\varphi(\1)=\1$ and
if the inflations $\varphi_n:=\varphi\otimes\mathrm{id}_n\colon\cA\otimes M_n(\C)
\to \cB\otimes M_n(\C)$ are positive for all $n\geq 1$.}
(see also \cite{Voiculescu93}).
This has become by now the
standard definition of quasidiagonality for operator algebras (see,
for example, \cite[Definition~7.1.1]{bBrown08}):

\begin{defn}
\label{def:abstractqd} A unital separable C*-algebra $\cA$ is called
{\em quasidiagonal} if there exists a sequence of u.c.p. maps
$\varphi _n \colon \cA\to M_{k(n)}(\mathbb C)$ which is both
asymptotically multiplicative (i.e., $\| \varphi_n (AB) -\varphi_n
(A)\varphi_n (B) \| \to 0$ for all $A,B\in \cA$) and asymptotically
isometric (i.e., $\| A \| =\lim _{n\to \infty} \| \varphi_n (A) \| $
for all $A\in \cA$).
\end{defn}

Inspired by Voiculescu's work on quasidiagonality we introduce in
this section an abstract definition of a F\o lner C*-algebra and
formulate our main result characterizing F\o lner C*-algebras in terms
of F\o lner sequences and also of amenable traces.

Recall that $\mathrm{tr}(\cdot)$ denotes the unique tracial state on a matrix algebra
$M_{n}(\C)$.
\begin{defn}\label{def:FA}
Let $\cA$ be a unital, separable C*-algebra.
\begin{itemize}
 \item[(i)] We say that $\cA$ is a {\em F\o lner C*-algebra} if there exists a sequence of u.c.p. maps
$\varphi_n\colon\cA\to M_{k(n)}(\C)$ such that
\begin{equation}\label{eq:mult-2}
\lim_n\|\varphi_n(AB)-\varphi_n(A)\varphi_n(B)\|_{2,\mathrm{tr}}=0\;,\quad A,B\in\cA\;,
\end{equation}
where $\|F\|_{2,\mathrm{tr}}:=\sqrt{\mathrm{tr}(F^*F)}$, $F\in M_{n}(\C)$\;.
 \item[(ii)]  We say that $\cA$ is a {\em proper F\o lner C*-algebra} if there exists a sequence of u.c.p. maps
  $\varphi_n\colon\cA\to M_{k(n)}(\C)$ satisfying (\ref{eq:mult-2}) and which, in addition,
  are asymptotically isometric, i.e.,
\begin{equation}\label{eq:norm}
   \|A\|=\lim_n\|\varphi_n(A)\|\;,\quad A\in\cA\;.
\end{equation}
\end{itemize}
\end{defn}

It is clear that if $\cA$ is a separable, unital and quasidiagonal
C*-algebra (cf.~Definition \ref{def:abstractqd}), then $\cA$ is a
proper F\o lner algebra. The Toeplitz algebra serves as a
counter-example to the reverse implication.

Although, in principle, the two concepts--F\o lner and proper 
F\o lner--seem to be different for C*-algebras, we can show that they indeed define the
same class of unital, separable C*-algebras. The proof of the next
proposition includes a useful trick so that we will include it here
(cf.~\cite[Proposition~3.2]{AraLledo12}).

\begin{prop}
\label{prop:F=PF} Let $\cA$ be a unital separable C*-algebra. Then
$\cA$ is a F\o lner C*-algebra if and only if $\cA$ is a proper
F\o lner C*-algebra.
\end{prop}
\begin{proof}
Assume that $\cA$ is a F\o lner C*-algebra, and let $\varphi_n
\colon \cA \to M_{k(n)}(\C )$ be a sequence of u.c.p. maps such that
(\ref{eq:mult-2}) holds. Considering the direct sum of a
sufficiently large number of copies of $\varphi_n$, for each $n$, we
may assume that
\begin{equation}
\label{eq:(n)to0} \lim _{n\to \infty} \frac{n}{k(n)}= 0.
\end{equation}
Let $\pi \colon \cA \to \cL (\cH )$ be a faithful representation of
$\cA$ on a separable Hilbert space $\cH$. Let $\{ P_n\}_n$ be an
increasing sequence of orthogonal projections on $\cH$, converging
to $\1$ in the strong operator topology and such that $\dim
P_n(\cH ) =n$ for all $n$. Then for all $A\in \cA$ we have $\| A\|
=\lim _n \| P_n\pi (A) P_n \| $.  Let $\psi _n \colon \cA \to
M_{k(n)+n}(\C) $ be given by:
$$
\psi _n (A) = \varphi _n (A) \oplus P_n \pi (A) P_n , \qquad A\in \cA.
$$
Then $\psi _n$ is a u.c.p. map. For $A,B\in \cA$,
set $X_n= P_n \pi (A)(1-P_n)\pi (B) P_n $. Then we have
\begin{align*}
\| \psi _n (AB) -\psi _n(A) & \psi _n(B)  \| _{2, \mathrm{tr}}^2 \le
\| \varphi _n (AB) -\varphi _n(A) \varphi _n(B) \| _{2,
\mathrm{tr}}^2
+ \frac{ \mathrm{Tr} (X_n^*X_n)}{k(n)+n}\\
& \le \| \varphi _n (AB) -\varphi _n(A) \varphi _n(B) \| _{2,
\mathrm{tr}}^2 + \frac{n \|A\|^2  \|B\|^2}{k(n)+n} \;.
\end{align*}
Using (\ref{eq:(n)to0}) we get
\begin{equation*}
\lim_n\|\psi_n(AB)-\psi_n(A)\psi_n(B)\|_{2,\mathrm{tr}}=0.
\end{equation*}
On the other hand, for $A\in \cA$, we have
$$
\|A \| -\|\psi _n (A) \|  \le  \| A \| - \| P_n\pi (A)P_n\| \to 0
$$
so that (\ref{eq:norm}) holds for the sequence $(\psi _n )$. This
concludes the proof.
\end{proof}

For the next result recall that a representation $\pi$ of an
abstract C*-algebra $\cA$ on a Hilbert space $\cH$ is called {\em
essential} if $\pi(\cA)$ contains no nonzero compact operators.
The proof uses the same approximation technique as the proof of
Theorem~\ref{teo:apply} (see Theorem~3.4 in \cite{AraLledo12}
for details).

\begin{thm} \label{thm:charFoe} Let $\cA$ be a unital
separable C*-algebra. Then the following conditions are equivalent:
\begin{itemize}
 \item[(i)] There exists a faithful representation $\pi\colon\cA\to\cL(\cH)$ such that $\pi(\cA)$ has a
  F\o lner sequence.
 \item[(ii)] There exists a faithful essential representation $\pi\colon\cA\to\cL(\cH)$ such that $\pi(\cA)$ has a
  F\o lner sequence.
 \item[(iii)] Every faithful essential representation $\pi\colon\cA\to\cL(\cH)$ satisfies that $\pi(\cA)$ has a
 proper F\o lner sequence.
 \item[(iv)] There exists a non-zero representation $\pi\colon\cA\to\cL(\cH)$ such that $\pi(\cA)$ has an amenable trace.
 \item[(v)] Every faithful representation $\pi\colon\cA\to\cL(\cH)$ satisfies that $\pi(\cA)$ has an ame\-nable trace.
 \item[(vi)] $\cA$ is a F\o lner C*-algebra.
\end{itemize}
 \end{thm}

\begin{rem}
\begin{itemize}
 \item[(i)]
The class of C*-algebras introduced in this section has been considered before by B\'edos.
In \cite{Bedos95} the author defines a C*-algebra $\cA$ to be {\em weakly hypertracial}
if $\cA$ has a non-degenerate representation
$\pi$ such that $\pi(\cA)$ has a hypertrace.
In this sense, the preceding theorem gives a new characterization of weakly hypertracial
C*-algebras in terms of u.c.p. maps.
 \item[(ii)] Note also that the equivalences between (i), (iv) and (v)
in Theorem~\ref{thm:charFoe} are basically known
(see \cite{Bedos95}).
\end{itemize}

\end{rem}

We conclude mentioning that in
the study of growth properties of C*-algebras (and motivated
by previous work done by Arveson and B\'edos)
Vaillant defined the following
natural unital C*-algebra (see Section~3 in \cite{Vaillant96}):
given an increasing sequence $\cP:=\{P_n\}_n\subset\cPfin(\cH)$
of orthogonal finite rank projections strongly converging to
$\1$, consider the set of all bounded
linear operators in $\cH$ that have $\cP$ as a proper F\o lner sequence, i.e.,
\[
 \cF_{\cP}(\cH):=\left\{
                 X\in \cL(\cH) \mid
                 \lim_{n\to\infty} \frac{\|X P_n-P_n X\|_2}{\|P_n\|_2} = 0
                 \right\}\;.
\]
This unital C*-subalgebra of $\cL(\cH)$ (called F\o lner algebra by Hagen, Roch
and Silbermann in Section~7.2.1 of \cite{bHagen01}) has shown to be very
useful in the analysis of the classical Szeg\"o limit theorems for
Toeplitz operators and some generalizations of them
(see, e.g., Section~7.2 of \cite{bHagen01} and \cite{Boettcher05}).

The C*-algebra $\cF_{\cP}$ is always non-separable for the operator
norm. Indeed, consider the $\ell^\infty$-direct sum of matrix
algebras $\cA=\prod_i M_{n_i}(\C)$, where $n_i$ are the ranks of the
orthogonal projections $P_{i+1}-P_i$, $i\in\N$,  with norm given by
$\| (a_i)\|=\sup_i \|a_i\|$. It is clear that $\cA$ is not
separable, and the elements of $\cA$ can be seen inside $\cF_{\cP}$
as block-diagonal operators, so the algebra $\cF _{\cP}$ is also
non-separable.

%%%%%%%%%%%%%%%%%%%%%%%%%%%%%%%%%%%%%%%%%%%%%%%%%%%%%%%%%%%%%%%%%%%%%%%%%%%%%%%%%%%%%%%%%%%%%%%%%%%%%%%%%
\section{Final remarks: F\o lner versus proper F\o lner}\label{sec:conclusion}

As was mentioned at the beginning of Section~\ref{sec:basic}, F\o lner sequences 
appeared first in the context of groups. Note that if
countable discrete group $\Gamma$ has a F\o lner sequence one can
always find another F\o lner sequence which, in addition to
Eq.~(\ref{eq:foelner-g}), is also proper, i.e.,
$\Gamma_i\subset\Gamma_j$ if $i\leq j$ and $\Gamma=\cup_i \Gamma_i$.
In the context of operators and due to the linear structure of
the underlying Hilbert spaces the difference between F\o lner sequence and
proper F\o lner sequence is relevant.
As was mentioned after Proposition~\ref{prop:decomp} if
$T= T_0\oplus T_1$ is a finite block reducible operator
on the Hilbert space $\cH=\cH_0\oplus\cH_1$, with $1\le \dim \cH_0<\infty$, and
$T_1$ strongly non-F\o lner (cf.~Subsection~\ref{sbsc:strg-NF}), then 
$T$ has an obvious constant F\o lner sequence but can not have a proper 
F\o lner sequence. Moreover, Proposition~\ref{prop:splitting} shows that
the difference between F\o lner and proper F\o lner sequence for single
operators can only appear in the case when there is a 
non-trivial finite dimensional invariant subspace.

At the level of abstract C*-algebras Proposition~\ref{prop:F=PF} shows that
F\o lner C*-algebras and proper F\o lner C*-algebras define the same class of
unital separable C*-algebras. Note that by Theorem~\ref{thm:charFoe}~(i)
the direct sum of a matrix algebra and the Cuntz algebra
\[
 \cA:=M_n(\C)\oplus\cO_n
\]
is a F\o lner (hence proper F\o lner) C*-algebra. But in its natural 
representation on $\cH:=\C^n\oplus\ell_2$ this algebra can not have
a proper F\o lner sequence because the representation is not essential
(see Theorem~\ref{thm:charFoe}~(iii)).

Finally, if $\cB$ is a {\em unital} C*-subalgebra of a F\o lner C*-algebra $\cA$,
then one can restrict the u.c.p.~maps of $\cA$ to $\cB$ to show that $\cB$ is also 
a F\o lner C*-algebra. This is not true if $\cB$ is a non-unital 
C*-subalgebra (i.e., if $\1_{\cA}\notin\cB$). Consider, for example, the concrete
C*-algebra on $\cK:=\ell_2\oplus\cH$ given by
\[
 \cA:=C^*(S)\oplus C^*(T_1)\;,
\]
where $S$ is the unilateral shift and $T_1$ is a strongly non-F\o lner operator.
Then, again, $\cA$ is a F\o lner (hence proper F\o lner) C*-algebra, but the 
non-unital C*-subalgebra $\cB:=0\oplus C^*(T_1)$ is not a F\o lner C*-algebra.

%%%%%%%%%%%%%%%%%%%%%%%%%%%%%%%%%%%%%%%%%%%%%%%%%%%%%%%%%%%%%%%%%%%%%%%%%%%%%%%%%%%%%%%%%%%%%%%%%%%%%%%%%%%%
\subsection*{Acknowledgment}
The two first-named authors would like to acknowledge the {\em Instituto Superior T\'ec\-nico, Lisboa,}
for the invitation and hospitality during the {\em Workshop on Operator Theory and Operator Algebras} 2012.

%%%%%%%%%%%%%%%%%%%%%%%%%%%%%%%%%%%%%%%%%%%%%%%%%%%%%%%%%%%%%%%%%%%%%%%%%%%%%%%%%%%%%%%%%%%%%%%%%%%%%%%%%%%%
%%%%%%%%%%%%%%%%%%%%%%%%%%%%%%%%%%%%%%%%%%%%%%%%%%%%%%%%%%%%%%%%%%%%%%%%%%%%%%%%%%%%%%%%%%%%%%%%%%%%%%%%%%%%

% ------------------------------------------------------------------------
\end{document}